\documentclass[11pt]{amsart}
\usepackage[T1]{fontenc}
\usepackage{lmodern,amsmath,amssymb,amsthm,mathtools}
\usepackage[margin=1in]{geometry}
\usepackage[expansion=false]{microtype}
\usepackage{tikz}
\usetikzlibrary{arrows.meta}
\usepackage{nicematrix}
\usepackage{placeins}
\usepackage{comment}
\usetikzlibrary{positioning,calc,shapes.geometric}
\usepackage[hidelinks]{hyperref}
\hypersetup{pdftitle={A cubic obstruction to the Bilu--Linial signing bound: an elementary construction and proof}}
\newtheorem{theorem}{Theorem}[section]
\newtheorem{conjecture}{Conjecture}[section]
\newtheorem{proposition}[theorem]{Proposition}
\newtheorem{lemma}[theorem]{Lemma}
\theoremstyle{remark}

\newcommand{\norm}[1]{\left\lVert#1\right\rVert}
\newcommand{\R}{\mathbb R}

\newcommand{\rr}{\rho_0}
\newcommand{\mya}{\alpha}
\definecolor{branchblue}{RGB}{35,87,130}
\definecolor{leafgreen}{RGB}{40,116,80}
\definecolor{completionorange}{RGB}{182,88,20}
\tikzset{vertex/.style={circle,fill=black,inner sep=1.65pt},
  leaf/.style={circle,draw=leafgreen,fill=white,line width=.8pt,inner sep=2pt},
  added/.style={circle,draw=completionorange,fill=completionorange!20,
    line width=.8pt,inner sep=2pt},
  block/.style={draw=branchblue,fill=branchblue!5,rounded corners=3pt,
    minimum width=1.5cm,minimum height=.65cm},
  edge/.style={line width=.65pt},
  newedge/.style={draw=completionorange,line width=.75pt}}

\title{A $3$-regular counterexample to the Bilu--Linial signing conjecture}

\author{Zhiqiang Xu}
\address{State Key Laboratory of Mathematical Sciences, Academy of Mathematics and Systems Science, Chinese Academy of Sciences, Beijing 100190, China }
\email{xuzq@lsec.cc.ac.cn}
\date{}
\begin{document}
\begin{abstract}
We construct a finite connected simple cubic graph $F$ such that
every signing of its edges yields a signed adjacency matrix
with an eigenvalue outside $[-2\sqrt2,2\sqrt2]$.
This disproves the Bilu--Linial signing conjecture for general
regular graphs. The graph $F$ is not Ramanujan, and the conjecture
restricted to Ramanujan base graphs remains open.
\end{abstract}
\maketitle
\raggedbottom

\section{Introduction}
\subsection{Problem setup}
Throughout this paper, all graphs are finite, simple, and undirected,
unless explicitly stated otherwise.
A graph is called $d$-regular if every vertex has exactly $d$ neighbours,
and a $3$-regular graph is called \emph{cubic}.
For a graph $G$, its vertex and edge sets are denoted by $V(G)$ and
$E(G)$. A \emph{signing} is a function
$\sigma:E(G)\to\{-1,1\}$. The signed adjacency matrix
$A_\sigma(G)$ is indexed by $V(G)$ and is defined by
\[
 (A_\sigma(G))_{uv}=
 \begin{cases}
  \sigma(\{u,v\}),&\{u,v\}\in E(G),\\
  0,&\{u,v\}\notin E(G).
 \end{cases}
\]
Here $\{u, v\}$ denotes the unordered edge; all diagonal entries
are zero. We write $A(G)$ for the unsigned adjacency matrix, obtained
by taking all edge signs to be $+1$.
For a real symmetric matrix $A$, its Euclidean operator norm is
\[
 \norm{A}=\max_{\boldsymbol{x}\ne\boldsymbol{0}}\frac{\norm{A\boldsymbol{x}}_2}{\norm{\boldsymbol{x}}_2}
          =\max\{|\lambda|:\lambda\text{ is an eigenvalue of }K\},
 \qquad \norm{\boldsymbol{x}}_2^2=\sum_i x_i^2.
\]
The vector in the maximum ranges over the real coordinate space on
which $K$ acts, and the sum ranges over all its coordinates.
Expanding in an orthonormal eigenbasis also gives
\[
 \norm{A}=\max_{\boldsymbol{x}\ne\boldsymbol{0}}
  \frac{|\boldsymbol{x}^{\mathsf T}A\boldsymbol{x}|}{\norm{\boldsymbol{x}}_2^2},
\]
where ${}^{\mathsf T}$ denotes transpose. The quotient without the
absolute value is the \emph{Rayleigh quotient} of $\boldsymbol{x}$ with respect
to $A$.

Bilu and Linial~\cite{BL04,BL06} asked whether every $d$-regular graph 
admits an assignment of signs $+1$ or $-1$ to its edges such that all
 eigenvalues of the resulting signed adjacency matrix lie in the interval 
 $[-2\sqrt{d-1},2\sqrt{d-1}]$. They formulated the following conjecture
  (see also Conjectures~6.9 and~6.10 in \cite{HLW06}).

\begin{conjecture}[Bilu--Linial {\cite[Conjecture~3.1]{BL06}}]
\label{conj:bilu-linial}
For every integer $d\ge2$ and every finite simple $d$-regular
graph $G$, there exists a signing
$\sigma:E(G)\to\{-1,1\}$ such that
\begin{equation}\label{eq:bl-bound}
  \norm{A_\sigma(G)}\,\,\le\,\,2\sqrt{d-1}.
\end{equation}
\end{conjecture}

The signing question arises naturally from graph coverings.
The \emph{$2$-lift} associated with a signing replaces each vertex $u$ by two vertices, denoted by $u_+$ and $u_-$. Each positive edge $uv$ is replaced by the two edges $u_+v_+$ and $u_-v_-$, while each negative edge $uv$ is replaced by the two edges $u_+v_-$ and $u_-v_+$.
Its adjacency eigenvalues, counted with multiplicity, are those of
$A(G)$ together with those of $A_\sigma(G)$
\cite[Lemma~3.1]{BL06}. 
Thus, \eqref{eq:bl-bound} requires all new eigenvalues to lie in the interval $[-2\sqrt{d-1},2\sqrt{d-1}]$.
 We refer to $G$ as the \emph{base graph}.

For a connected $d$-regular graph, the eigenvalue $d$, and also $-d$
when the graph is bipartite, are called trivial. 
A $d$-regular graph is called {\em Ramanujan} 
if all of its nontrivial eigenvalues lie in the interval
 $[-2\sqrt{d-1},2\sqrt{d-1}]$ \cite{rdef}.
This definition clarifies the role of \eqref{eq:bl-bound}: if the base
graph is Ramanujan, the bound ensures that all new eigenvalues lie
in the Ramanujan interval, so the lift is also Ramanujan.
For a comprehensive account of the spectral properties,
constructions, and applications of expander graphs,
see the survey \cite{HLW06}.

\subsection{Our contribution}

We construct a  $3$-regular graph such that, for every edge signing,
the resulting signed adjacency matrix has at least one eigenvalue
outside the interval $[-2\sqrt2,2\sqrt2]$.
This provides a counterexample to the Bilu--Linial conjecture.

\begin{theorem}\label{thm:main}
There exists a finite, connected, simple, $3$-regular graph $F$
such that every signing $\sigma:E(F)\to\{-1,1\}$ satisfies
\[
  \norm{A_\sigma(F)}>2\sqrt2.
\]
\end{theorem}

We describe the construction of $F$ in Section~2 and prove in the subsequent sections that it satisfies the conclusion of Theorem~\ref{thm:main}.

\subsection{Related work}

Marcus, Spielman, and Srivastava
\cite[Theorem~5.3]{MSS15} proved that every $d$-regular graph,
with $d\ge2$, admits a signing whose signed adjacency eigenvalues
are all at most $2\sqrt{d-1}$.
For bipartite graphs, the signed adjacency spectrum is symmetric
about zero, so this result establishes the Bilu--Linial conjecture
in the bipartite case.
Their proof introduced the method of interlacing families of
polynomials, which they subsequently developed to resolve the
Kadison--Singer problem \cite{MSS15KS}.
Further developments extended this approach to higher-rank
versions of Weaver's $\mathrm{KS}_r$ conjecture, with improved
quantitative bounds obtained by Xu, Xu, and Zhu \cite{XXZ23}.

Bilu and Linial~\cite{BL06} proved that every graph of maximum degree $d\geq 2$ admits a signing whose signed adjacency matrix has spectral radius $O(\sqrt{d (\log d)^3} )$.
For $d$-regular graphs, Ravichandran and Srivastava~\cite{RavichandranSrivastava2021} improved this bound to $2\sqrt{2(d-1)}$ using interlacing families of mixed determinantal polynomials.
An alternative proof of the latter bound, based on a
graph-theoretic reduction to the bipartite signing theorem,
is given in \cite{Huang26}.
Lin and Zhou \cite{LZ26} further improved this bound, proving
that every finite $d$-regular graph $G$, with $d\ge2$, admits
a signing $\sigma$ such that
\(
  \norm{A_\sigma(G)}
  < \frac{3+\sqrt5}{2}\sqrt{d-1}.
\)

The covering problem extends naturally to $\ell$-lifts.
For an integer $\ell\ge2$, an $\ell$-lift replaces each vertex
by $\ell$ copies and each edge by a perfect matching between
the corresponding sets of copies. Its spectrum contains that
of the base graph; the remaining eigenvalues, counted with
multiplicity, are called the new eigenvalues.
Hall, Puder, and Sawin \cite{HPS18} proved that, for every
integer $\ell\ge2$, every connected $d$-regular graph with
$d\ge3$ admits an $\ell$-lift whose new eigenvalues are at most
$2\sqrt{d-1}$. For a bipartite base graph, spectral symmetry
gives the corresponding lower bound. Thus every connected
bipartite Ramanujan graph admits a Ramanujan $\ell$-lift
for every integer $\ell\ge2$.

\subsection{Organization}
Section~\ref{sec:construction} describes the construction of the
cubic graph $F$.
Section~\ref{sec:lemmas} states the auxiliary lemmas used in
the main proof, with proofs provided in
Sections~\ref{sec:elimination} and~\ref{sec:seed}.
Section~\ref{sec:main-proof} uses these lemmas to prove
Theorem~\ref{thm:main}.
Section~\ref{sec:scope} shows that $F$ is not Ramanujan.
Thus our counterexample disproves the conjecture for general
regular graphs but leaves its restriction to Ramanujan base
graphs unresolved. The final section presents several open
questions for further research.

\section{The graph and its construction}\label{sec:construction}

A \emph{rooted graph} is a graph with a distinguished vertex, called its \emph{root}. Our construction proceeds in three stages. We first define a rooted graph $H_{62}$, which we call the \emph{seed}. We then join copies of the seed to form the \emph{core} $J$ and complete this core to obtain the $3$-regular counterexample $F$. Throughout the construction, all graph copies are taken to be pairwise vertex-disjoint before the specified connecting edges are added. In particular, repeated occurrences of the same graph in a recursive formula denote distinct copies.

\subsection{The rooted seed $H_{62}$}

For each nonnegative integer $h$, let $T_h$ be the complete rooted
binary tree of height $h$. The level of a vertex is its distance
from the root, measured in edges. Thus, the root is at level zero,
each vertex at a level less than $h$ has exactly two children,
and the vertices at level $h$ have no children. In particular,
$T_0$ consists of a single vertex. To specify the construction
unambiguously, we label the two children of each branching vertex
as its first and second children.

We construct the rooted graph $H_h$ as follows. Start with a triangle
on vertices $v_0,v_1,v_2$, add a new vertex $o$, and join $o$ to $v_0$.
Next, take three vertex-disjoint copies of $T_h$, also disjoint
from these four vertices, and join their roots to $v_1$, $v_2$,
and $o$, respectively, by one edge each. We designate $o$ as
the root of $H_h$.

The root $o$ has degree two. The vertices at level $h$ in the
three attached trees have degree one in $H_h$ and are precisely
its \emph{leaves}, of which there are $3\cdot 2^h$.
Every other vertex has degree three. We use $H_{62}$ as the
seed for the counterexample; its structure is illustrated
schematically in Figure~\ref{fig:construction}(a).

\subsection{Joining copies of the seed to form the core $J$}

For two vertex-disjoint rooted graphs $X$ and $Y$, let $B(X,Y)$
be the rooted graph obtained by adding a new vertex, joining it
to the root of each graph $X, Y$ by an edge, and taking this new vertex
as the root. Starting from the seed $H_{62}$, define recursively
\begin{equation}\label{eq:branches}
  Q_0=H_{62},\qquad
  Q_{j+1}=B(Q_j,Q_j)\quad(0\le j<61),
\end{equation}
where the two occurrences of $Q_j$ denote vertex-disjoint copies.
The root of each $Q_j$ has degree two. At each recursive step,
the roots of the two copies acquire one additional edge and
therefore have degree three in $Q_{j+1}$.

To form the \emph{core} $J$, take three vertex-disjoint copies
of $Q_{61}$, add a new vertex $z$, and join $z$ to the root
of each copy. The vertex $z$ and the three roots then all have
degree three. Consequently, every vertex of $J$ has degree three
except the leaves of the binary trees within the seed copies,
which retain degree one.
Figure~\ref{fig:construction}(b) illustrates the recursive step
$Q_{j+1}=B(Q_j,Q_j)$, and Figure~\ref{fig:construction}(c)
shows the structure of the core $J$.

\subsection{Completing the core to obtain the $3$-regular counterexample $F$}
We complete the core $J$ by applying the following operation
separately to each copy of the seed $H_{62}$. First, order the
leaves by listing those in the tree attached to $v_1$, followed
by those in the tree attached to $v_2$, and then those in the
tree attached to $o$. Within each tree, order the leaves
lexicographically according to the sequences of first and second
child choices along their paths from the root, with the first
child preceding the second. Partition the resulting list into
consecutive triples; this is possible because each seed has
$3\cdot 2^{62}$ leaves.

For each triple, introduce two new vertices and join each of them
to all three leaves in that triple. All newly introduced vertices
are distinct, and no other edges are added. In particular, the
two vertices introduced for a given triple are not adjacent.
Let $F$ denote the graph obtained from $J$ after this operation
has been performed in every seed copy.

Each leaf of $J$ gains exactly two incident edges and therefore
has degree three in $F$. Each new vertex also has degree three,
while the degrees of all other vertices remain unchanged.
Thus, $F$ is $3$-regular. Since the new vertices are distinct
and each is joined to three distinct leaves, the construction
introduces neither loops nor multiple edges. Moreover, $F$ is
connected because $J$ is connected and every new vertex is
adjacent to vertices of $J$. Hence $F$ is a finite, connected,
simple, $3$-regular graph.
Finally, $J$ is an induced subgraph of $F$, since the completion
adds no edges between vertices of $J$.
Figure~\ref{fig:construction} summarizes the construction.

\begin{figure}[tbp]\label{fig:F}
\centering
\begin{minipage}[t]{.48\textwidth}
\centering
\textbf{(a) The seed $H_h$}\par\smallskip
\begin{tikzpicture}[x=1cm,y=1cm,font=\small]
 \node[vertex,label=above:$o$] (o) at (0,1.3) {};
 \node[vertex,label=left:$v_0$] (v0) at (0,.15) {};
 \node[vertex,label=left:$v_1$] (v1) at (-1,-1) {};
 \node[vertex,label=right:$v_2$] (v2) at (1,-1) {};
 \draw[edge] (o)--(v0)--(v1)--(v2)--(v0);
 \node[block,minimum width=1cm] (t0) at (1.7,1.3) {$T_h$};
 \node[block,minimum width=1cm] (t1) at (-1,-2.1) {$T_h$};
 \node[block,minimum width=1cm] (t2) at (1,-2.1) {$T_h$};
 \draw[edge] (o)--(t0.west) (v1)--(t1.north) (v2)--(t2.north);
 \node[align=center,font=\footnotesize] at (.35,-2.9)
 {Each box is a whole rooted tree.\\The joining edge ends at its root.};
\end{tikzpicture}
\end{minipage}\hfill
\begin{minipage}[t]{.48\textwidth}
\centering
\textbf{(b) One binary join}\par\smallskip
\begin{tikzpicture}[x=1cm,y=1cm,font=\small]
 \node[vertex] (o) at (0,1.3) {};
 \node[above=.12cm of o] {new root};
 \node[block] (x) at (-1.35,-.05) {$Q_j$};
 \node[block] (y) at (1.35,-.05) {$Q_j$};
 \draw[edge] (o)--(x.north) (o)--(y.north);
 \node[align=center] at (0,-1.15) {$Q_{j+1}=B(Q_j,Q_j)$};
 \node[align=center,font=\footnotesize] at (0,-2.15)
 {The two boxes are disjoint copies.\\Their edge signings may differ.};
 \path (0,-2.9);
\end{tikzpicture}
\end{minipage}

\medskip
\begin{minipage}[t]{.48\textwidth}
\centering
\textbf{(c) The core $J$}\par\smallskip
\begin{tikzpicture}[x=1cm,y=1cm,font=\small]
 \node[vertex,label=above:$z$] (z) at (0,1.3) {};
 \node[block,minimum width=1.3cm] (x) at (-1.9,-.3) {$Q_{61}$};
 \node[block,minimum width=1.3cm] (y) at (0,-.3) {$Q_{61}$};
 \node[block,minimum width=1.3cm] (w) at (1.9,-.3) {$Q_{61}$};
 \draw[edge] (z)--(x.north) (z)--(y.north) (z)--(w.north);
 \node[align=center,font=\footnotesize] at (0,-1.35)
 {Each branch contains $2^{61}$ seeds.\\The central dot is one vertex.};
 \path (0,-2.2);
\end{tikzpicture}
\end{minipage}\hfill
\begin{minipage}[t]{.48\textwidth}
\centering
\textbf{(d) Completing one leaf triple}\par\smallskip
\begin{tikzpicture}[x=1cm,y=1cm,font=\small]
 \node[added,label=above:new] (a) at (0,1.2) {};
 \node[added,label=below:new] (b) at (0,-1.2) {};
 \node[leaf] (l) at (-1.5,0) {};
 \node[leaf] (m) at (0,0) {};
 \node[leaf] (n) at (1.5,0) {};
 \draw[newedge] (a)--(l)--(b) (a)--(m)--(b) (a)--(n)--(b);
 \draw[edge] (l)--(-2.2,0) (n)--(2.2,0)
 (m)--(.55,.1)--(.8,.15);
 \node[align=center,font=\footnotesize] at (0,-2.05)
 {Open circles are old leaves.\\Black stubs are their original parent edges.};
\end{tikzpicture}
\end{minipage}
\caption{The exact construction, with repeated subgraphs compressed
into boxes. In (a), set $h=62$; in (b), perform the joins for
$j=0,\ldots,60$. After (c), carry out (d) separately in each seed,
using the leaf ordering specified in the text. Orange vertices and
edges are added only at the completion step.}
\label{fig:construction}
\end{figure}
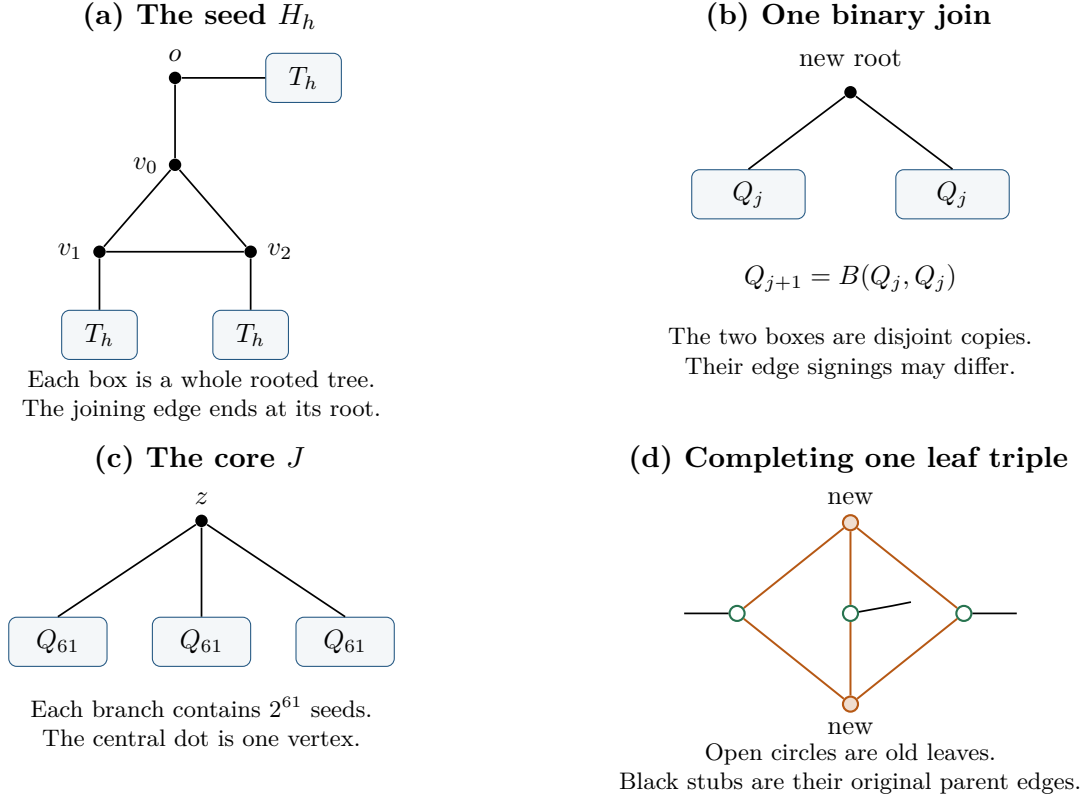

\subsection{A small example illustrating the construction}

To illustrate the construction with every vertex shown explicitly,
Figure~\ref{fig:small} presents a smaller graph obtained by joining
the roots of three copies of $H_0$ to a central vertex and then
applying the same completion rule. The resulting graph is
$3$-regular, with $28$ vertices and $42$ edges. It serves only
to illustrate the construction and we make no claim that it satisfies
the spectral conclusion of Theorem~\ref{thm:main}.

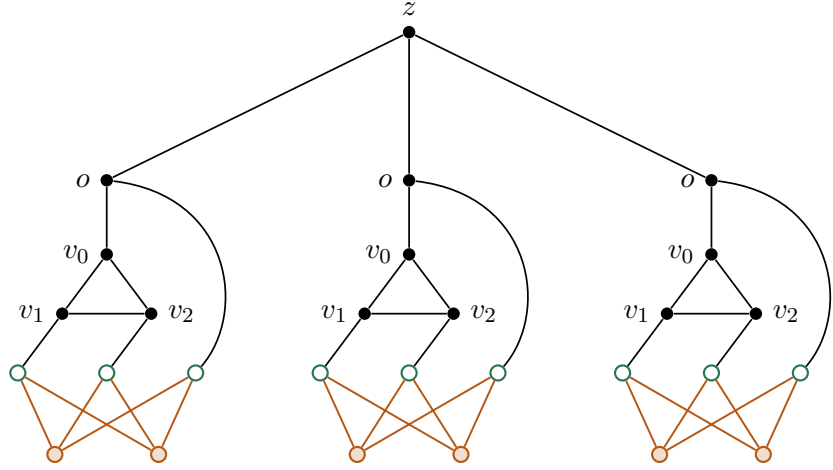
\begin{figure}[tbp]
\centering
\begin{tikzpicture}[x=.98cm,y=.98cm]
 \node[vertex,label=above:$z$] (z) at (0,2) {};
 \foreach \s/\xx in {L/-4,M/0,R/4}{
  \begin{scope}[xshift=\xx cm]
   \node[vertex,label=left:$o$] (o\s) at (0,0) {};
   \node[vertex,label=left:$v_0$] (v\s) at (0,-1) {};
   \node[vertex,label=left:$v_1$] (a\s) at (-.6,-1.8) {};
   \node[vertex,label=right:$v_2$] (b\s) at (.6,-1.8) {};
   \node[leaf] (l\s) at (-1.2,-2.6) {};
   \node[leaf] (m\s) at (0,-2.6) {};
   \node[leaf] (n\s) at (1.2,-2.6) {};
   \node[added] (p\s) at (-.7,-3.7) {};
   \node[added] (q\s) at (.7,-3.7) {};
   \draw[edge] (o\s)--(v\s)--(a\s)--(b\s)--(v\s)
    (a\s)--(l\s) (b\s)--(m\s);
   \draw[edge] (o\s) .. controls (1.7,-.2) and (1.9,-1.8) .. (n\s);
   \draw[newedge] (p\s)--(l\s) (p\s)--(m\s) (p\s)--(n\s)
    (q\s)--(l\s) (q\s)--(m\s) (q\s)--(n\s);
  \end{scope}
  \draw[edge] (z)--(o\s);
 }
\end{tikzpicture}
\caption{A $28$-vertex illustration obtained by reducing the tree
height to zero and omitting the binary joins. Black dots are seed
vertices or the centre, green open circles are old leaves, and
orange circles are the six completion vertices. All vertices and
edges are shown. Edge crossings without a circle are not vertices.}
\label{fig:small}
\end{figure}

\FloatBarrier

\section{Definitions and lemmas for the main proof}\label{sec:lemmas}

\subsection{Schur complements}
Schur complements play a key role in our proof.
We recall their definition and the standard properties needed below;
see \cite[Appendix~A.5.5, p.~650]{BV04}.

Let $\ell_1,\ell_2$ be positive integers, and let
\[
  \mathcal M=
  \begin{pmatrix}
    D & W\\
    W^{\mathsf T} & \Gamma
  \end{pmatrix}
  \in\mathbb R^{(\ell_1+\ell_2)\times(\ell_1+\ell_2)}
\]
be a real symmetric matrix, where
$D\in\mathbb R^{\ell_1\times\ell_1}$,
$\Gamma\in\mathbb R^{\ell_2\times\ell_2}$, and
$W\in\mathbb R^{\ell_1\times\ell_2}$.
If $D$ is invertible, the matrix
\[
  \Gamma-W^{\mathsf T}D^{-1}W
  \in\mathbb R^{\ell_2\times\ell_2}
\]
is called the \emph{Schur complement of $D$ in $\mathcal M$}.
Similarly, if $\Gamma$ is invertible, the \emph{Schur complement
of $\Gamma$ in $\mathcal M$} is
\(
  D-W\Gamma^{-1}W^{\mathsf T}
  \in\mathbb R^{\ell_1\times\ell_1}.
\)

The Schur complement of $D$ arises naturally when eliminating
the variables corresponding to the block $D$.
Indeed, consider the system
\[
  D\boldsymbol{\xi}+W\boldsymbol{\eta}=\boldsymbol{0},
  \qquad
  W^{\mathsf T}\boldsymbol{\xi}+\Gamma\boldsymbol{\eta}=\boldsymbol{b},
\]
where $\boldsymbol{\xi}\in\mathbb R^{\ell_1}$ and
$\boldsymbol{\eta},\boldsymbol{b}\in\mathbb R^{\ell_2}$.
Since $D$ is invertible, the first equation gives
$\boldsymbol{\xi}=-D^{-1}W\boldsymbol{\eta}$. Substituting this expression into the
second equation yields the reduced system
\[
  (\Gamma-W^{\mathsf T}D^{-1}W)\boldsymbol{\eta}=\boldsymbol{b}.
\]

The following lemma is standard and we include a proof for completeness.
For a real symmetric matrix $K\in \R^{n\times n}$, the notation
$K\succeq0$ means $\boldsymbol{x}^{\mathsf T}K\boldsymbol{x}\ge0$ for every real vector $\boldsymbol{x}\in {\mathbb R}^n$,
and $K\succ0$ means strict inequality for every nonzero $\boldsymbol{x}\in {\mathbb R}^n$.
\begin{lemma}\label{lem:schur}
With the notation above, assume in addition that $D\succ0$.
Then $\mathcal M\succeq0$ if and only if
$\Gamma-W^{\mathsf T}D^{-1}W\succeq0$. The analogous equivalence holds
with $\succ0$ throughout. 
If $\mathcal M\succ0$, the lower-right block of $\mathcal M^{-1}$
is $(\Gamma-W^{\mathsf T}D^{-1}W)^{-1}$.
\end{lemma}
\begin{proof}

Since $D\succ0$, the matrix $D$ is invertible. For
$\boldsymbol{\xi}\in\R^{\ell_1}$ and $\boldsymbol{\eta}\in\R^{\ell_2}$, we have
\begin{equation}\label{eq:etaxi}
\begin{aligned}
\begin{pmatrix}\boldsymbol{\xi}\\ \boldsymbol{\eta}\end{pmatrix}^{\mathsf T}
\mathcal M
\begin{pmatrix}\boldsymbol{\xi}\\ \boldsymbol{\eta}\end{pmatrix}
&=\boldsymbol{\xi}^{\mathsf T}D\boldsymbol{\xi}
  +2\boldsymbol{\xi}^{\mathsf T}W\boldsymbol{\eta}+\boldsymbol{\eta}^{\mathsf T}\Gamma\boldsymbol{\eta}\\
&=(\boldsymbol{\xi}+D^{-1}W\boldsymbol{\eta})^{\mathsf T}
  D(\boldsymbol{\xi}+D^{-1}W\boldsymbol{\eta})\\
&\quad+\boldsymbol{\eta}^{\mathsf T}
  (\Gamma-W^{\mathsf T}D^{-1}W)\boldsymbol{\eta}.
\end{aligned}
\end{equation}
The first term in the last expression is nonnegative and vanishes
exactly when $\boldsymbol{\xi}=-D^{-1}W\boldsymbol{\eta}$.

Suppose first that $\mathcal M\succeq0$. Choosing
$\boldsymbol{\xi}=-D^{-1}W\boldsymbol{\eta}$ in (\ref{eq:etaxi}) gives
\[
\boldsymbol{\eta}^{\mathsf T}(\Gamma-W^{\mathsf T}D^{-1}W)\boldsymbol{\eta}\ge0
\qquad\text{for every }\boldsymbol{\eta}\in\R^{\ell_2}.
\]
Thus the Schur complement is positive semidefinite.
If $\mathcal M\succ0$, the same choice gives strict inequality
whenever $\boldsymbol{\eta}\ne\boldsymbol{0}$, because the combined vector
$(\boldsymbol{\xi},\boldsymbol{\eta})$ is then nonzero. Hence the Schur complement
is positive definite.

Conversely, if the Schur complement is positive semidefinite,
both terms in (\ref{eq:etaxi}) are nonnegative,
so $\mathcal M\succeq0$. 
Here, we use $D \succ 0$.
If the Schur complement is positive
definite, consider any nonzero vector $(\boldsymbol{\xi},\boldsymbol{\eta})$.
When $\boldsymbol{\eta}\ne\boldsymbol{0}$, the second term is strictly positive.
When $\boldsymbol{\eta}=\boldsymbol{0}$, we have $\boldsymbol{\xi}\ne\boldsymbol{0}$, and the first term is
$\boldsymbol{\xi}^{\mathsf T}D\boldsymbol{\xi}>0$. Therefore $\mathcal M\succ0$.

It remains to prove the inverse formula. Assume that
$\mathcal M\succ0$, so both $\mathcal M$ and its Schur complement
are invertible. For an arbitrary $\boldsymbol{\zeta}\in\R^{\ell_2}$, write
\[
\begin{pmatrix}\boldsymbol{\xi}\\ \boldsymbol{\eta}\end{pmatrix}
=\mathcal M^{-1}\begin{pmatrix}\boldsymbol{0}\\ \boldsymbol{\zeta}\end{pmatrix},
\]
where the zero vector belongs to $\R^{\ell_1}$.
Equivalently,
\[
D\boldsymbol{\xi}+W\boldsymbol{\eta}=\boldsymbol{0},
\qquad
W^{\mathsf T}\boldsymbol{\xi}+\Gamma\boldsymbol{\eta}=\boldsymbol{\zeta}.
\]
The first equation gives $\boldsymbol{\xi}=-D^{-1}W\boldsymbol{\eta}$.
Substituting into the second yields
\[
(\Gamma-W^{\mathsf T}D^{-1}W)\boldsymbol{\eta}=\boldsymbol{\zeta},
\qquad
\boldsymbol{\eta}=(\Gamma-W^{\mathsf T}D^{-1}W)^{-1}\boldsymbol{\zeta}.
\]
By definition, the lower-right block of $\mathcal M^{-1}$
maps $\boldsymbol{\zeta}$ to $\boldsymbol{\eta}$. Since $\boldsymbol{\zeta}$ was arbitrary,
this block is precisely
$(\Gamma-W^{\mathsf T}D^{-1}W)^{-1}$.
\end{proof}

\subsection{Root responses and spectral bounds}

Throughout the rest of the paper, set $\rr:=2\sqrt2$.
For a graph $X$, write $\R^{V(X)}$ for the space of
real-valued functions on its vertex set, and let $I_X$
denote the identity matrix on this space.
By the spectral theorem,
\begin{equation}\label{eq:norm-psd}
\norm{A_\sigma(X)}\le\rr
\quad\text{if and only if}\quad
\rr I_X-A_\sigma(X)\succeq0
\quad\text{and}\quad
\rr I_X+A_\sigma(X)\succeq0.
\end{equation}

Suppose that $X$ is rooted at $o_X$ and that both matrices
in \eqref{eq:norm-psd} are positive definite.
Let $\boldsymbol{e}_{o_X}\in\R^{V(X)}$ be the vector whose coordinate
at $o_X$ is $1$ and whose other coordinates are $0$.
Define
\begin{equation}\label{eq:response}
\begin{aligned}
g_\varepsilon(X,\sigma)
&=\boldsymbol{e}_{o_X}^{\mathsf T}
  (\rr I_X-\varepsilon A_\sigma(X))^{-1}\boldsymbol{e}_{o_X},
&&\varepsilon\in\{-1,1\},\\
s(X,\sigma)
&=g_1(X,\sigma)+g_{-1}(X,\sigma).
\end{aligned}
\end{equation}
We call $g_\varepsilon(X,\sigma)$ the \emph{root response}
on side $\varepsilon$. Each root response is strictly
positive, since it is a diagonal entry of a positive
definite inverse. The two responses defining $s(X,\sigma)$
use the same signing $\sigma$.
When a signing is defined on a larger graph, we also use
$\sigma$ for its restriction to the edges of a subgraph.

The reciprocal of the root response $g_\varepsilon(X,\sigma)$ has the following
Schur complement interpretation.

\begin{proposition}\label{prop:response-schur}
Let $X$ be a rooted graph with at least two vertices,
and let $\sigma$ be a signing for which both matrices
in \eqref{eq:norm-psd} are positive definite.
Fix $\varepsilon\in\{-1,1\}$ and order the vertices of $X$
with the root $o_X$ last. Write
\[
\rr I_X-\varepsilon A_\sigma(X)
=
\begin{pmatrix}
D_\varepsilon & \boldsymbol{b}_\varepsilon\\
\boldsymbol{b}_\varepsilon^{\mathsf T} & \rr
\end{pmatrix},
\]
where $D_\varepsilon$ is the principal submatrix indexed
by $V(X)\setminus\{o_X\}$, and the last row and column
correspond to $o_X$. Then
\[
\frac{1}{g_\varepsilon(X,\sigma)}
=
\rr-\boldsymbol{b}_\varepsilon^{\mathsf T}
D_\varepsilon^{-1}\boldsymbol{b}_\varepsilon>0.
\]
Thus $1/g_\varepsilon(X,\sigma)$ is the scalar Schur
complement obtained by eliminating all nonroot vertices.
\end{proposition}

\begin{proof}
Since the full matrix is positive definite, so is
$D_\varepsilon$. By Lemma~\ref{lem:schur}, its Schur
complement is positive, and the bottom-right entry of
the inverse of the full matrix is
\[
\left(\rr-\boldsymbol{b}_\varepsilon^{\mathsf T}
D_\varepsilon^{-1}\boldsymbol{b}_\varepsilon\right)^{-1}.
\]
This entry equals $g_\varepsilon(X,\sigma)$ by
\eqref{eq:response}, which proves the assertion.
\end{proof}

If $X$ consists only of its root, then
$g_\varepsilon(X,\sigma)=1/\rr$, so its reciprocal
is simply $\rr$.

\subsection{Auxiliary lemmas}

The proof of the main result uses the following three lemmas.
Lemmas~\ref{lem:branches} and~\ref{lem:join} are proved in
Section~\ref{sec:elimination}, and Lemma~\ref{prop:seed} is proved
in Section~\ref{sec:seed}.

\begin{lemma}
\label{lem:branches}
Let $\sigma$ be a signing of $J$ such that
$\norm{A_\sigma(J)}\le \rr=2\sqrt{2}$.
For each $j\in\{0,\ldots,61\}$ and each rooted copy $X$ of $Q_j$
appearing in the construction of $J$, we have
\[
\rr I_X-\varepsilon A_\sigma(X)\succ0
\qquad\text{for every }\varepsilon\in\{-1,1\},
\]
where $\sigma$ on $X$ denotes the restriction of the signing
to $E(X)$.
In particular, both root responses $g_1(X,\sigma)$ and
$g_{-1}(X,\sigma)$, and hence their sum $s(X,\sigma)$,
are well defined by \eqref{eq:response},
even when $\norm{A_\sigma(J)}=\rr$.
\end{lemma}

\begin{lemma}\label{lem:join}
Let $X$ and $Y$ be vertex-disjoint rooted graphs, and let
$Z=B(X,Y)$ be the rooted graph obtained by adding a new vertex
and joining this vertex by an edge to each of the roots of $X$ and $Y$.
The new vertex is designated as the root of $Z$.
Suppose that a signing $\sigma$ of $Z$ satisfies
\[
\rr I_Z-\varepsilon A_\sigma(Z)\succ0
\qquad\text{for each }\varepsilon\in\{-1,1\}.
\]
Use the same notation $\sigma$ for its restrictions to $X$ and $Y$.
Then the root responses of $X$, $Y$, and $Z$ defined in
\eqref{eq:response} are well defined. Moreover, for each
$\varepsilon\in\{-1,1\}$, 
\[
p_\varepsilon
:=p_\varepsilon(X,Y,\sigma)
:=\rr-g_\varepsilon(X,\sigma)-g_\varepsilon(Y,\sigma)\,\,>\,\,0,
\]
and
\begin{equation}\label{eq:join-response}
\begin{aligned}
g_\varepsilon(Z,\sigma)&=\frac1{p_\varepsilon},
\qquad \varepsilon\in\{-1,1\},\\
s(Z,\sigma)&\ge
\frac{4}{2\rr-s(X,\sigma)-s(Y,\sigma)}.
\end{aligned}
\end{equation}
The denominator in the second inequality equals
$p_1+p_{-1}$ and is therefore strictly positive.
\end{lemma}

\begin{lemma}\label{prop:seed}
For every signing $\sigma$ of $H_{62}$,
$\rr I_{H_{62}}-\varepsilon A_\sigma(H_{62})\succ0$ for
$\varepsilon\in\{-1,1\}$, and
\begin{equation}\label{eq:seed-value}
 \frac{s(H_{62},\sigma)}{\sqrt2}
   =\frac{140300416}{138131009}>\frac{65}{64}.
\end{equation}
\end{lemma}

\section{Proof of Theorem~\ref{thm:main}}\label{sec:main-proof}

We first outline the main ideas of the proof.
The proof proceeds by contradiction on the induced core $J$.
Assuming that a signing satisfies $\norm{A_\sigma(J)}\le \rr:=2\sqrt{2}$,
Lemma~\ref{lem:branches} guarantees that all branch matrices remain
positive definite even when the assumed spectral bound is attained
with equality. Hence these matrices are invertible, and their root
responses $g_\epsilon, \epsilon\in \{-1, 1\}$ and the corresponding Schur complements are well defined.
Lemma~\ref{prop:seed} provides an initial lower bound for the
sum of the two root responses of each seed, and
Lemma~\ref{lem:join} propagates this bound through the binary
joins.
After $61$ levels, each of the three branches attached to the centre
has response sum greater than $4\sqrt2/3$, so the sum of their
response sums exceeds $4\sqrt2$.
However, taking Schur complements at the centre in
$\rr I_J-A_\sigma(J)$ and $\rr I_J+A_\sigma(J)$ forces the total
of these three response sums to be at most $2\rr=4\sqrt2$,
a contradiction. Since $J$ is an induced subgraph of $F$,
the resulting strict norm bound also holds for every signing
of $F$.

\begin{proof}[Proof of Theorem~\ref{thm:main}]

By the construction in Section~2, the graph $F$ is finite,
connected, simple, and cubic. We now prove that
$\|A_\sigma(F)\|>\rr$ for every signing $\sigma$ of $F$.

It suffices  to consider $J$. 
Since $J$ is an induced subgraph of $F$, every signing $\sigma$
of $F$ restricts to a signing of $J$, and $A_\sigma(J)$ is a
principal submatrix of $A_\sigma(F)$. Hence
\[
\|A_\sigma(J)\|\le \|A_\sigma(F)\|.
\]
It therefore suffices to prove that every signing of $J$ has
norm strictly greater than $\rr$.
Suppose, to the contrary, that
a signing $\sigma$ satisfies $\norm{A_\sigma(J)}\le \rr$.
By Lemma~\ref{lem:branches}, both root responses
$g_1(X,\sigma)$ and $g_{-1}(X,\sigma)$ are defined for every
induced copy $X$ of $Q_j$ in $J$, with $0\le j\le61$,
even when $\norm{A_\sigma(J)}=\rr$.

For this proof, define
\[
 c_j=\sqrt2\left(1+\frac1{64-j}\right),
 \qquad 0\le j\le61.
\]
We claim that, for each $0\le j\le61$ and every copy $X$ of
$Q_j$ appearing in the construction of $J$, equipped with the
restriction of $\sigma$ to its edges,
\begin{equation}\label{eq:induction}
s(X,\sigma)>c_j.
\end{equation}
Here, $s(X,\sigma)$ is defined in  (\ref{eq:response}).
We first complete the main argument using this claim and prove the claim afterward.

Let $X_1,X_2,X_3$ be the three copies of $Q_{61}$ whose roots
are joined to the central vertex $z$. Fix
$\varepsilon\in\{-1,1\}$, and write
\[
D_i=\rr I_{X_i}-\varepsilon A_\sigma(X_i),
\qquad i=1,2,3.
\]
Each $D_i$ is positive definite by Lemma~\ref{lem:branches}.
Order the vertices of $J$ by listing the vertices of
$X_1,X_2,X_3$ first and $z$ last. Since there are no edges
between these three branches, we have
\[
\rr I_J-\varepsilon A_\sigma(J)
=
\begin{pmatrix}
D_1&0&0&\boldsymbol{b}_1\\
0&D_2&0&\boldsymbol{b}_2\\
0&0&D_3&\boldsymbol{b}_3\\
\boldsymbol{b}_1^{\mathsf T}&\boldsymbol{b}_2^{\mathsf T}&\boldsymbol{b}_3^{\mathsf T}&\rr
\end{pmatrix}
\succeq0,
\]
where
\[
\boldsymbol{b}_i=-\varepsilon\,\sigma(zo_{X_i})\boldsymbol{e}_{o_{X_i}}
\]
records the single edge joining $z$ to the root $o_{X_i}$ of
$X_i$. Applying Lemma~\ref{lem:schur} to the positive definite
block $\operatorname{diag}(D_1,D_2,D_3)$ shows that
\[
\rr-\sum_{i=1}^3 \boldsymbol{b}_i^{\mathsf T}D_i^{-1}\boldsymbol{b}_i\ge0.
\]
The sign of each joining edge disappears from this expression,
because its square is one. Thus, by the definition of the root
response,
\[
\boldsymbol{b}_i^{\mathsf T}D_i^{-1}\boldsymbol{b}_i
=\boldsymbol{e}_{o_{X_i}}^{\mathsf T}D_i^{-1}\boldsymbol{e}_{o_{X_i}}
=g_\varepsilon(X_i,\sigma).
\]
We therefore obtain
\[
\rr-\sum_{i=1}^3g_\varepsilon(X_i,\sigma)\ge0
\qquad\text{for each }\varepsilon\in\{-1,1\}.
\]
Adding these two inequalities and using
$s(X_i,\sigma)=g_1(X_i,\sigma)+g_{-1}(X_i,\sigma)$ gives
\[
0\le2 \rr-\sum_{i=1}^3s(X_i,\sigma).
\]
However, \eqref{eq:induction} at $j=61$ gives
$s(X_i,\sigma)>c_{61}=4\sqrt2/3$ for each $i=1,2,3$. Hence
\[
0\le2\rr-\sum_{i=1}^3s(X_i,\sigma)
<4\sqrt2-3\left(\frac{4\sqrt2}{3}\right)=0,
\]
a contradiction. Thus every signing of $J$ has norm strictly
greater than $\rr$. We arrive at the conclusion.

It remains to prove \eqref{eq:induction}.
We prove \eqref{eq:induction} by induction on $j$.
Lemma~\ref{prop:seed} proves the case $j=0$ since $Q_0=H_{62}$.
Suppose it holds at level $j<61$, and let $X,Y$ be the two child
branches of an occurrence $Z$ of $Q_{j+1}$. The two restrictions
of $\sigma$ can be different, but each satisfies the claimed bound.
Lemma~\ref{lem:join} gives
\[
 s(Z,\sigma)
 \ge\frac4{\,2\rr-s(X,\sigma)-s(Y,\sigma)\,}
 >\frac2{\rr-c_j}.
\]
Lemma~\ref{lem:join} gives
$2\rr-s(X,\sigma)-s(Y,\sigma)>0$.
The other denominator is also positive, since $0\le j<61$ implies
$64-j\ge4$, and hence
\[
\rr-c_j
=\sqrt2\left(1-\frac1{64-j}\right)
\ge\frac{3\sqrt2}{4}>0.
\]
Direct simplification gives
\begin{equation}\label{eq:closed-induction}
 \frac2{\rr-c_j}
 =\sqrt2\,\frac{64-j}{63-j}
 =\sqrt2\left(1+\frac1{63-j}\right)=c_{j+1}.
\end{equation}
This proves \eqref{eq:induction} for every required level.

\end{proof}

\section{Proofs of Lemmas~\ref{lem:branches} and \ref{lem:join}}\label{sec:elimination}

In this section, we prove Lemmas~\ref{lem:branches} and~\ref{lem:join}.
These results establish the positive definiteness of the branch
matrices and the response inequality for joining two rooted graphs,
which are the key ingredients in the proof of the main theorem.

\begin{proof}[Proof of Lemma~\ref{lem:branches}]
The key point is that every branch has an edge joining its root
to a parent vertex outside the branch. This edge allows us to
deduce strict positivity for the branch matrix from positive
semidefiniteness of the matrix on $J$.

Fix $\varepsilon\in\{-1,1\}$. The assumption
$\norm{A_\sigma(J)}\le \rr=2\sqrt{2}$ implies
\[
\rr I_J-\varepsilon A_\sigma(J)\succeq0.
\]
We prove by induction on $j$ that
$\rr I_X-\varepsilon A_\sigma(X)\succ0$ for every rooted copy $X$
of $Q_j$ appearing in $J$.

For $j=0$, each such copy is a seed $H_{62}$, so the assertion
follows from Lemma~\ref{prop:seed}. 

Suppose the assertion holds at level $j$, where $0\le j<61$,
and let $X$ be a rooted copy of $Q_{j+1}$ in $J$.
By construction, $X=B(X_1,X_2)$, where $X_1$ and $X_2$ are
vertex-disjoint rooted copies of $Q_j$. Denote their roots by
$o_1$ and $o_2$, and denote the new root of $X$ by $w$.
The induction hypothesis gives
\[
D_i:=\rr I_{X_i}-\varepsilon A_\sigma(X_i)\succ0,
\qquad i=1,2.
\]

Let $v$ be the unique neighbor of $w$ outside $X$.
Such a vertex exists for every branch under consideration:
it is the next join vertex toward the centre of $J$, or the
centre $z$ itself when $X$ is a copy of $Q_{61}$.
Moreover, $v$ has no neighbors in $X_1$ or $X_2$.
The relevant connections are illustrated in
Figure~\ref{fig:branch-parent}.

For $i=1,2$, define the column vector
\[
\boldsymbol{b}_i=-\varepsilon\sigma(\{w,o_i\})\boldsymbol{e}_{o_i}
\in\R^{V(X_i)},
\]
where $\boldsymbol{e}_{o_i}$ is the coordinate vector at the root $o_i$.
Order the vertices in $V(X)\cup\{v\}$ by listing those of
$X_1$, those of $X_2$, and then $w,v$.
 The corresponding principal
submatrix of $\rr I_J-\varepsilon A_\sigma(J)$ is
\[
\begin{pmatrix}
D_1 & 0   & \boldsymbol{b}_1 & \boldsymbol{0}\\
0   & D_2 & \boldsymbol{b}_2 & \boldsymbol{0}\\
\boldsymbol{b}_1^{\mathsf T} & \boldsymbol{b}_2^{\mathsf T} & \rr & a\\
\boldsymbol{0} & \boldsymbol{0} & a & \rr
\end{pmatrix}
\succeq0,
\]
where
\(
a:=-\varepsilon\sigma(\{w,v\}).
\)
Since $D_1,D_2\succ0$, Lemma~\ref{lem:schur} allows us to
eliminate these two blocks. The resulting Schur complement is
\[
\begin{pmatrix}
p&a\\
a&\rr
\end{pmatrix}
\succeq0,
\qquad
p:=\rr-\boldsymbol{b}_1^{\mathsf T}D_1^{-1}\boldsymbol{b}_1
     -\boldsymbol{b}_2^{\mathsf T}D_2^{-1}\boldsymbol{b}_2.
\]
Its determinant must be nonnegative. Therefore
\[
p\rr-a^2=p\rr-1\ge0,
\]
which implies
\(
p\ge\frac1\rr>0.
\)

Now consider the matrix on $X=B(X_1,X_2)$ alone:
\[
\rr I_X-\varepsilon A_\sigma(X)
=
\begin{pmatrix}
D_1 & 0   & \boldsymbol{b}_1\\
0   & D_2 & \boldsymbol{b}_2\\
\boldsymbol{b}_1^{\mathsf T} & \boldsymbol{b}_2^{\mathsf T} & \rr
\end{pmatrix}.
\]
The Schur complement of $\operatorname{diag}(D_1,D_2)$
is the same scalar $p$. 
Since $D_1$ and $D_2$ are positive definite and $p>0$,
Lemma~\ref{lem:schur} gives
\[
\rr I_X-\varepsilon A_\sigma(X)\succ0.
\]
This completes the induction.

The argument applies separately to $\varepsilon=1$ and
$\varepsilon=-1$, with the same signing $\sigma$.
Thus both branch matrices are positive definite for every
required copy $X$, and all the stated root responses are
well defined.
\end{proof}

\begin{figure}[htbp]
\centering
\begin{tikzpicture}[
    x=1cm,
    y=1cm,
    vertex/.style={
        circle,
        fill=black,
        inner sep=1.8pt
    },
    branch/.style={
        draw=black!60,
        fill=black!3,
        rounded corners=5pt,
        line width=0.6pt
    },
    edge/.style={
        line width=0.8pt
    }
]

\draw[dashed, rounded corners=7pt, line width=0.7pt]
    (-3.2,-2.1) rectangle (3.2,1.65);

\node[anchor=north east] at (3.07,1.55)
    {$X=B(X_1,X_2)$};

\draw[branch]
    (-2.8,-1.8) rectangle (-0.8,0.35);
\draw[branch]
    (0.8,-1.8) rectangle (2.8,0.35);

\coordinate (v)  at (0,2.7);
\coordinate (w)  at (0,1.05);
\coordinate (o1) at (-1.8,0);
\coordinate (o2) at (1.8,0);

\draw[edge] (v) -- (w);
\draw[edge] (w) -- (o1);
\draw[edge] (w) -- (o2);

\node[vertex, label=left:{$v$}] at (v) {};
\node[vertex, label=left:{$w$}] at (w) {};
\node[vertex, label=left:{$o_1$}] at (o1) {};
\node[vertex, label=right:{$o_2$}] at (o2) {};

\node at (-1.8,-0.65) {$\vdots$};
\node at (1.8,-0.65) {$\vdots$};

\node at (-1.8,-1.35) {$X_1\cong Q_j$};
\node at (1.8,-1.35) {$X_2\cong Q_j$};

\node[anchor=west, font=\small] at (0.3,2.7)
    {exterior parent};

\end{tikzpicture}
\caption{The branch $X=B(X_1,X_2)$ and its exterior parent $v$.
The child branches $X_1$ and $X_2$ are vertex-disjoint copies
of $Q_j$, rooted at $o_1$ and $o_2$, respectively.
The new root $w$ is joined to $o_1$, $o_2$, and $v$.
There are no edges between $X_1$ and $X_2$, and $v$ has
no neighbors in either child branch.
The dashed boundary encloses $X$; adjoining $v$ gives the
vertex set indexing the principal submatrix used in the proof.
When $X$ is a copy of $Q_{61}$, its exterior parent is $v=z$.
The internal structures of the child branches are omitted.}
\label{fig:branch-parent}
\end{figure}
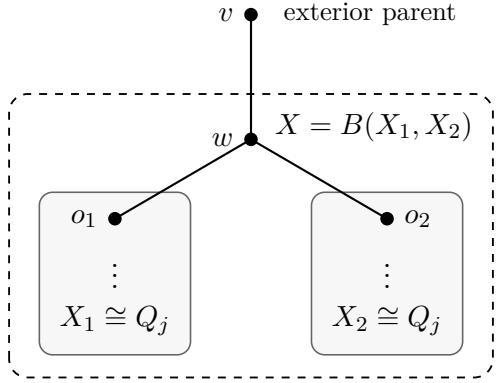

\begin{proof}[Proof of Lemma~\ref{lem:join}]
We first compute the response at the new root by block elimination,
and then combine the formulas for the two spectral sides.

Let $o_X$ and $o_Y$ be the roots of $X$ and $Y$, respectively,
and let $o_Z$ be the new root of $Z=B(X,Y)$.
Recall that $B(X,Y)$ is obtained from the disjoint union of
$X$ and $Y$ by adding a new vertex $o_Z$ and the two edges
$\{o_Z,o_X\}$ and $\{o_Z,o_Y\}$, with $o_Z$ designated as its root.
 Denote the signs
of the two joining edges by
\[
\tau_X=\sigma(\{o_Z,o_X\}),
\qquad
\tau_Y=\sigma(\{o_Z,o_Y\}).
\]
Thus $\tau_X,\tau_Y\in\{-1,1\}$.

Fix $\varepsilon\in\{-1,1\}$ and set
\[
D_X=\rr I_X-\varepsilon A_\sigma(X),
\qquad
D_Y=\rr I_Y-\varepsilon A_\sigma(Y).
\]
Both matrices are positive definite, since they are principal
submatrices of the positive definite matrix
$\rr I_Z-\varepsilon A_\sigma(Z)\succ 0$. In particular, their inverses
exist and the corresponding root responses $g_\varepsilon$ are well defined.

Order the vertices of $Z$ by listing those of $X$ first,
those of $Y$ second, and $o_Z$ last. Since the only edges
between these three parts are the two joining edges, we have
\[
\rr I_Z-\varepsilon A_\sigma(Z)
=
\begin{pmatrix}
D_X & 0 & -\varepsilon\tau_X \boldsymbol{e}_{o_X}\\
0 & D_Y & -\varepsilon\tau_Y \boldsymbol{e}_{o_Y}\\
-\varepsilon\tau_X \boldsymbol{e}_{o_X}^{\mathsf T}
&-\varepsilon\tau_Y \boldsymbol{e}_{o_Y}^{\mathsf T}
&\rr
\end{pmatrix}.
\]
Here $\boldsymbol{e}_{o_X}$ and $\boldsymbol{e}_{o_Y}$ are the coordinate vectors
at the respective roots.

Apply Lemma~\ref{lem:schur} to eliminate the positive definite
block $\operatorname{diag}(D_X,D_Y)$. Its Schur complement
is the scalar
\[
\begin{aligned}
&\rr-(-\varepsilon\tau_X)^2
     \boldsymbol{e}_{o_X}^{\mathsf T}D_X^{-1}\boldsymbol{e}_{o_X}
   -(-\varepsilon\tau_Y)^2
     \boldsymbol{e}_{o_Y}^{\mathsf T}D_Y^{-1}\boldsymbol{e}_{o_Y}\\
&\qquad
=\rr-g_\varepsilon(X,\sigma)-g_\varepsilon(Y,\sigma)
=p_\varepsilon>0,
\end{aligned}
\]
where we used $\varepsilon^2=\tau_X^2=\tau_Y^2=1$.
Since the full matrix is positive definite, Lemma~\ref{lem:schur}
implies that its Schur complement $p_\varepsilon>0$.
 According to 
Proposition \ref{prop:response-schur},
\begin{equation}\label{eq:g1gm1}
g_\varepsilon(Z,\sigma)
=\boldsymbol{e}_{o_Z}^{\mathsf T}
 (\rr I_Z-\varepsilon A_\sigma(Z))^{-1}\boldsymbol{e}_{o_Z}
=\frac1{p_\varepsilon}.
\end{equation}

The argument applies to both $\varepsilon=1$ and $\varepsilon=-1$,
using the same signing $\sigma$. Adding the definitions of
$p_1$ and $p_{-1}$ gives
\begin{equation}\label{eq:p1p-1}
p_1+p_{-1}
=2\rr-s(X,\sigma)-s(Y,\sigma)>0.
\end{equation}
For these two positive numbers,
\begin{equation}\label{eq:p1p-12}
\frac1{p_1}+\frac1{p_{-1}}
-\frac4{p_1+p_{-1}}
=
\frac{(p_1-p_{-1})^2}
     {p_1p_{-1}(p_1+p_{-1})}
\ge0.
\end{equation}
Combining 
(\ref{eq:g1gm1}), 
(\ref{eq:p1p-1}) and (\ref{eq:p1p-12}), we have
\[
\begin{aligned}
s(Z,\sigma)
&=g_1(Z,\sigma)+g_{-1}(Z,\sigma)=\frac1{p_1}+\frac1{p_{-1}}\\
&\ge\frac4{p_1+p_{-1}}
=\frac4{2\rr-s(X,\sigma)-s(Y,\sigma)}.
\end{aligned}
\]
This proves the response formula, the stated inequality,
and the positivity of its denominator.
\end{proof}

\section{Proof of Lemma~\ref{prop:seed}}\label{sec:seed}

We begin with an auxiliary lemma that provides the spectral bounds
needed in the proof of Lemma~\ref{prop:seed}. These bounds follow
from known results for unsigned trees and unicyclic graphs
\cite{Stevanovic03,Hu07}, but we include an elementary proof for
completeness. We then compute the root response of the complete
binary tree explicitly and use this calculation to prove
Lemma~\ref{prop:seed}.

\begin{lemma}\label{lem:unicyclic}
Recall that $\rr:=2\sqrt{2}$.
The following statements hold.
\begin{enumerate}
\item[(i)]
Let $U$ be a finite connected graph with exactly one cycle
and maximum degree at most three. Then every signing $\sigma$
of $U$ satisfies
\[
\norm{A_\sigma(U)}<\rr.
\]

\item[(ii)]
Let $h$ be a nonnegative integer, and let $T_h$ be the
rooted binary tree defined in Section~2.1.
For every signing $\sigma$ of $T_h$ and every
$\varepsilon\in\{-1,1\}$, we have
\[
\rr I_{T_h}-\varepsilon A_\sigma(T_h)\succ0.
\]
Moreover, the corresponding root response is
\begin{equation}\label{eq:tree-response}
g_\varepsilon(T_h,\sigma)
=\frac{h+1}{(h+2)\sqrt2},
\qquad \varepsilon\in\{-1,1\}.
\end{equation}
\end{enumerate}
\end{lemma}

\begin{proof}

\noindent\emph{(i) Graphs with exactly one cycle.}
Let $\sigma$ be an arbitrary signing of $U$.
Write the vertices of its unique cycle in cyclic order as
$c_0,\ldots,c_{m-1}$, where $m\ge3$.
Orient the cycle edges from $c_i$ to $c_{i+1}$, with indices
taken modulo $m$.

Every connected component outside the unique cycle of $U$
contains no cycle and is therefore a tree. Since $U$ is
connected, each such tree is joined to the cycle by at
least one edge. There can be only one such edge, since
two distinct joining edges would create a second cycle
in $U$.
Orient the cycle edges consistently around the cycle and
all other edges away from the cycle. Each vertex then has
exactly one incoming edge: from the preceding vertex on
the cycle, or from its parent in an attached tree.
Since every vertex has degree at most three, at most two
edges point out of each vertex. We denote this number by
$d^+(v)$.
 An example of this orientation is shown in
Figure~\ref{fig:unicyclic-tree}(a).

Let $\boldsymbol{x}=(x_v)_{v\in V(U)}$ be a real vector.
For every directed edge $u\to v$, we have
\begin{equation}\label{eq:edge-inequality}
2|x_ux_v|
\le \frac{x_u^2}{\sqrt2}+\sqrt2\,x_v^2.
\end{equation}
Equality holds precisely when
\(
|x_u|=\sqrt2\,|x_v|.
\)

We now sum \eqref{eq:edge-inequality} over all edges.
Each vertex contributes $\sqrt2\,x_v^2$ through its unique
incoming edge and $d^+(v)x_v^2/\sqrt2$ through its outgoing
edges. Therefore
\begin{equation}\label{eq:quadratic-bound}
\begin{aligned}
|\boldsymbol{x}^{\mathsf T}A_\sigma(U)\boldsymbol{x}|
&=\left|2\sum_{\{u,v\}\in E(U)}
       \sigma(\{u,v\})x_ux_v\right|\\
&\le 2\sum_{\{u,v\}\in E(U)}|x_ux_v|\\
&\le \sum_{v\in V(U)}
       \left(\sqrt2+\frac{d^+(v)}{\sqrt2}\right)x_v^2\\
&\le \rr\|\boldsymbol{x}\|_2^2.
\end{aligned}
\end{equation}
The last inequality follows from $d^+(v)\le2$.
We claim that equality holds in \eqref{eq:quadratic-bound}
if and only if $\boldsymbol{x}=\boldsymbol{0}$. By compactness of the unit sphere,
this implies $\|A_\sigma(U)\|<\rr$.
We now prove the claim.

Observe that at least one vertex
has outdegree less than two. If $U$ is a cycle, every vertex
has outdegree one. Otherwise, an attached tree has a leaf,
whose outdegree is zero.

Suppose that equality holds in \eqref{eq:quadratic-bound}
for some nonzero vector $\boldsymbol{x}$.
Then equality must hold at every step of that chain.
 The final equality in
\eqref{eq:quadratic-bound} therefore gives
\[
\sum_{v\in V(U)}
\frac{2-d^+(v)}{\sqrt2}\,x_v^2=0.
\]
Every summand is nonnegative. Thus $x_v=0$ at every vertex
whose outdegree is less than two, and at least one such
vertex exists.

Equality must also hold in \eqref{eq:edge-inequality}
on every edge, since the sum of its nonnegative differences
is zero. Consequently,
\[
|x_u|=\sqrt2\,|x_v|
\qquad\text{whenever }u\to v.
\]
If either endpoint of an edge has zero coordinate, this
identity forces the other endpoint to have zero coordinate
as well. Since $U$ is connected, the zero coordinate found
above propagates along paths to every vertex. This gives
$\boldsymbol{x}=\boldsymbol{0}$, a contradiction.

\medskip
\noindent\emph{(ii) Full rooted binary trees.}
We prove the positive definiteness and the response formula
simultaneously by induction on the height $h$.

When $h=0$, the tree $T_0$ consists of its root alone.
Its adjacency matrix is zero, so for either
$\varepsilon\in\{-1,1\}$,
\[
\rr I_{T_0}-\varepsilon A_\sigma(T_0)\succ0,
\qquad
g_\varepsilon(T_0,\sigma)
=\frac1\rr=\frac1{2\sqrt2}.
\]
This is \eqref{eq:tree-response} at $h=0$.

Suppose the assertions hold at height $h$, and let $\sigma$
be any signing of $T_{h+1}$.
Denote its root by $o$. Deleting the root $o$ and its two incident edges from $T_{h+1}$
leaves two connected components, denoted by $X_1$ and $X_2$.
Each component is a copy of $T_h$, rooted at one of the two
children of $o$, which we denote by $o_1$ and $o_2$, respectively.
Thus $T_{h+1}=B(X_1,X_2)$, as illustrated in
Figure~\ref{fig:unicyclic-tree}(b).
We also write $\sigma$ for the restriction of the signing
to either child tree.

Fix $\varepsilon\in\{-1,1\}$ and define
\[
D_i:=\rr I_{X_i}-\varepsilon A_\sigma(X_i),
\qquad i=1,2.
\]
By the induction hypothesis, both matrices are positive
definite and
\begin{equation}\label{eq:gepdeng}
\boldsymbol{e}_{o_i}^{\mathsf T}D_i^{-1}\boldsymbol{e}_{o_i}
=g_\varepsilon(X_i,\sigma)
=\frac{h+1}{(h+2)\sqrt2},
\qquad i=1,2,
\end{equation}
where $\boldsymbol{e}_{o_i}$ is the coordinate vector at $o_i$.
The restrictions of $\sigma$ to the two child trees may
differ; the induction hypothesis applies to each of them.

Let
\[
\tau_i:=\sigma(\{o,o_i\})\in \{-1,1\},
\qquad
\boldsymbol{b}_i:=-\varepsilon\tau_i \boldsymbol{e}_{o_i}\in {\mathbb R}^{V(X_i)},
\qquad i=1,2.
\]
Ordering the vertices of $X_1$ first, those of $X_2$ second,
and $o$ last gives
\[
\rr I_{T_{h+1}}-\varepsilon A_\sigma(T_{h+1})
=
\begin{pmatrix}
D_1 & 0 & \boldsymbol{b}_1\\
0 & D_2 & \boldsymbol{b}_2\\
\boldsymbol{b}_1^{\mathsf T} & \boldsymbol{b}_2^{\mathsf T} & \rr
\end{pmatrix}.
\]
The Schur complement of $\operatorname{diag}(D_1,D_2)$
is the scalar
\[
\begin{aligned}
p_\varepsilon
&:=\rr-\boldsymbol{b}_1^{\mathsf T}D_1^{-1}\boldsymbol{b}_1
      -\boldsymbol{b}_2^{\mathsf T}D_2^{-1}\boldsymbol{b}_2\\
&=\rr-g_\varepsilon(X_1,\sigma)
     -g_\varepsilon(X_2,\sigma)\\
&=2\sqrt2-\frac{2(h+1)}{(h+2)\sqrt2}\\
&=\frac{(h+3)\sqrt2}{h+2}>0.
\end{aligned}
\]
The second equality follows from $(-\varepsilon\tau_i)^2=1$
for $i=1,2$, and the third follows from \eqref{eq:gepdeng}.

Since $D_1,D_2\succ0$ and $p_\varepsilon>0$,
Lemma~\ref{lem:schur} proves that
\[
\rr I_{T_{h+1}}-\varepsilon A_\sigma(T_{h+1})\succ0.
\]
According to Proposition \ref{prop:response-schur},
\[
g_\varepsilon(T_{h+1},\sigma)
=\frac1{p_\varepsilon}
=\frac{h+2}{(h+3)\sqrt2}.
\]
This is the required formula at height $h+1$.
Since $\sigma$ and $\varepsilon$ were arbitrary,
the induction proves both assertions for every height,
every signing, and both spectral sides.
\end{proof}

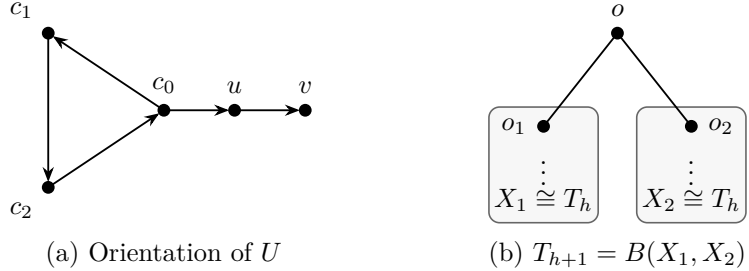
\begin{figure}[htbp]
\centering
\begin{tikzpicture}[
    x=0.85cm,
    y=0.85cm,
    font=\small,
    vertex/.style={
        circle,
        fill=black,
        inner sep=1.7pt
    },
    directed edge/.style={
        -{Stealth[length=2mm]},
        line width=0.7pt,
        shorten <=2pt,
        shorten >=2pt
    },
    tree edge/.style={
        line width=0.7pt
    },
    subtree/.style={
        draw=black!60,
        fill=black!3,
        rounded corners=4pt,
        line width=0.6pt
    }
]

\begin{scope}
    \coordinate (c0) at (0,0.6);
    \coordinate (c1) at (-1.8,1.8);
    \coordinate (c2) at (-1.8,-0.6);
    \coordinate (u)  at (1.1,0.6);
    \coordinate (v)  at (2.2,0.6);

    \draw[directed edge] (c0) -- (c1);
    \draw[directed edge] (c1) -- (c2);
    \draw[directed edge] (c2) -- (c0);

    \draw[directed edge] (c0) -- (u);
    \draw[directed edge] (u) -- (v);

    \node[vertex, label=above:{$c_0$}] at (c0) {};
    \node[vertex, label=above left:{$c_1$}] at (c1) {};
    \node[vertex, label=below left:{$c_2$}] at (c2) {};
    \node[vertex, label=above:{$u$}] at (u) {};
    \node[vertex, label=above:{$v$}] at (v) {};

    \node at (0,-1.65) {(a) Orientation of $U$};
\end{scope}

\begin{scope}[xshift=6cm]
    \draw[subtree]
        (-2,-1.15) rectangle (-0.3,0.65);
    \draw[subtree]
        (0.3,-1.15) rectangle (2,0.65);

    \coordinate (o)  at (0,1.8);
    \coordinate (o1) at (-1.15,0.35);
    \coordinate (o2) at (1.15,0.35);

    \draw[tree edge] (o) -- (o1);
    \draw[tree edge] (o) -- (o2);

    \node[vertex, label=above:{$o$}] at (o) {};
    \node[vertex, label=left:{$o_1$}] at (o1) {};
    \node[vertex, label=right:{$o_2$}] at (o2) {};

    \node at (-1.15,-0.25) {$\vdots$};
    \node at (1.15,-0.25) {$\vdots$};
    \node at (-1.15,-0.8) {$X_1\cong T_h$};
    \node at (1.15,-0.8) {$X_2\cong T_h$};

    \node at (0,-1.65) {(b) $T_{h+1}=B(X_1,X_2)$};
\end{scope}

\end{tikzpicture}
\caption{The two structures used in the proof.
(a) An example with a three-vertex cycle
$c_0c_1c_2c_0$ and an attached path $c_0uv$.
The cycle is oriented cyclically, and the attached path
is oriented away from it.
(b) The root $o$ of $T_{h+1}$ is joined to the roots
$o_1,o_2$ of two disjoint copies $X_1,X_2$ of $T_h$.
The shaded regions represent the child trees, whose
internal structures are omitted.
For $h=0$, each child tree consists only of its root.}
\label{fig:unicyclic-tree}
\end{figure}

We next present the proof of Lemma~\ref{prop:seed}.
\begin{proof}[Proof of Lemma~\ref{prop:seed}]

The graph $H_{62}$ is connected, has exactly one cycle, and has
maximum degree three. By Lemma~\ref{lem:unicyclic}, for every
signing $\sigma$ of $H_{62}$, both matrices
\[
  \rr I_{H_{62}}-A_\sigma(H_{62})
  \quad\text{and}\quad
  \rr I_{H_{62}}+A_\sigma(H_{62})
\]
are positive definite.

\emph{Reducing the signs.}
For notational convenience, write $S:=A_\sigma(H_{62})$
throughout this proof.
We simplify the edge signs by \emph{vertex switching}.
Specifically, choose an auxiliary sign $\delta_u\in\{-1,1\}$
at each vertex $u$ and replace the sign of each edge $uv$ by
$\delta_u\sigma(\{u,v\})\delta_v$.
The resulting signed adjacency matrix is $\Delta S\Delta$,
where
\(
  \Delta=\operatorname{diag}(\delta_u)_{u\in V(H_{62})}.
\)
Since $\Delta^{-1}=\Delta$, this transformation preserves
the spectrum. It also preserves the diagonal entries of
the inverses appearing in \eqref{eq:response}. Indeed,
whenever $\rr I_{H_{62}}-\varepsilon S$ is invertible,
\[
  (\rr I_{H_{62}}-\varepsilon\Delta S\Delta)^{-1}
  =\Delta(\rr I_{H_{62}}-\varepsilon S)^{-1}\Delta.
\]
Conjugation by $\Delta$ leaves each diagonal entry unchanged,
because $\delta_u^2=1$ for every vertex $u$.
In particular, vertex switching preserves the root response
$g_\varepsilon(H_{62},\sigma)$.

Deleting the edge $\{v_1,v_2\}$ from $H_{62}$ yields a tree
on the same vertex set.
 We choose $\Delta$ so that every edge of the spanning tree
has sign $+1$ after switching. Write $\delta_u$ for the
diagonal entry of $\Delta$ at a vertex $u$, and set $\delta_o=1$.
Proceeding from the root toward the leaves, define
\(
  \delta_v=\delta_u\,\sigma(\{u,v\})
\)
whenever $u$ is the parent of $v$ in the tree. Then the
transformed sign on this edge is
\(
  \delta_u\,\sigma(\{u,v\})\,\delta_v=1.
\)
This determines $\Delta$ on all vertices. 

In the switched signed adjacency matrix $\Delta S\Delta$,
every edge except possibly
$\{v_1,v_2\}$ has sign $+1$. The sign on $\{v_1,v_2\}$ is
\[
  \tau
  :=\sigma(\{v_0,v_1\})\sigma(\{v_1,v_2\})\sigma(\{v_2,v_0\})
  \in\{-1,1\}.
\]
Indeed, replacing $S$ by $\Delta S\Delta$ preserves the
product of the edge signs around the triangle: each diagonal entry of $\Delta$
associated with a triangle vertex appears twice in this product
and therefore contributes a factor of $1$.
Since the entries of $\Delta S\Delta$ corresponding to
$\{v_0,v_1\}$ and $\{v_0,v_2\}$ are both $+1$, its entries
corresponding to $\{v_1,v_2\}$ must equal $\tau$.

\emph{Eliminating the trees.}
Work first with $\rr I_{H_{62}}-\Delta S\Delta$ after this sign change.
Each of the three attached trees has response $63/(64\sqrt2)$
by \eqref{eq:tree-response}. In this calculation set
\[
 \mya=\rr-\frac{63}{64\sqrt2}=\frac{193}{64\sqrt2}.
\]
In the matrix $\rr I_{H_{62}}-\Delta S\Delta$, take the Schur
complement of the block diagonal principal submatrix indexed
by the vertices of the three attached copies of $T_{62}$.
This eliminates those tree vertices and retains only
$o,v_0,v_1,v_2$.  Wen use $M_\tau$ to denote the Schur complement.
Before eliminating the trees, the principal submatrix corresponding
to $o,v_0,v_1,v_2$, in this order, is
\[
 \begin{pNiceMatrix}[last-row,last-col]
    \rr  & -1 & 0     & 0     & o   \\
    -1 & \rr  & -1    & -1    & v_0 \\
    0  & -1 & \rr     & -\tau & v_1 \\
    0  & -1 & -\tau & \rr     & v_2 \\
    o  & v_0& v_1   & v_2   &
  \end{pNiceMatrix}.
\]
Since the three trees are disjoint and each is attached to only
one retained vertex, the correction term in the Schur complement is
\[
  \operatorname{diag}\left(
    \frac{63}{64\sqrt2},\,0,\,
    \frac{63}{64\sqrt2},\,\frac{63}{64\sqrt2}
  \right).
\]
Subtracting this term replaces the diagonal entries at
$o,v_1,v_2$ by $\mya$, while leaving the diagonal entry at $v_0$
and all off-diagonal entries unchanged. This gives
\begin{equation}\label{eq:four-matrix}
  M_\tau=
  \begin{pNiceMatrix}[last-row,last-col]
    \mya  & -1 & 0     & 0     & o   \\
    -1 & \rr  & -1    & -1    & v_0 \\
    0  & -1 & \mya     & -\tau & v_1 \\
    0  & -1 & -\tau & \mya     & v_2 \\
    o  & v_0& v_1   & v_2   &
  \end{pNiceMatrix}.
\end{equation}

For each $\tau\in\{-1,1\}$, the matrix $M_\tau$ arises from
a signing of the seed and is therefore positive definite.
Let $u_\tau$ denote the diagonal entry of $M_\tau^{-1}$
corresponding to $o$. By the inverse-block formula in Lemma~\ref{lem:schur} and
the invariance of the root response under switching,
\[
  g_1(H_{62},\sigma)=u_\tau=(M_\tau^{-1})_{11}.
\]
Indeed, by the inverse-block formula, the principal submatrix of
$(\rr I_{H_{62}}-\Delta S\Delta)^{-1}$ indexed by
$o,v_0,v_1,v_2$ is $M_\tau^{-1}$.
Since $o$ is the first vertex in this ordering,
$(M_\tau^{-1})_{11}$ is the diagonal entry of the full inverse
corresponding to the root $o$. This entry is the root response
of the switched signing and, by switching invariance, equals
$g_1(H_{62},\sigma)$.

A direct calculation using the cofactor formula gives
\begin{equation}\label{eq:seed-u}
  u_\tau
  =\frac{\rr\mya-2-\rr\tau}{\mya^2\rr-3\mya+(1-\mya\rr)\tau}.
\end{equation}

\emph{Adding the two spectral sides.}
Negating the original signing negates the product around its
three-edge cycle. After the same reduction, the negative signing
therefore has parameter $-\tau$. Consequently the two responses
of the original signing are $u_1,u_{-1}$, in some order, and
$s(H_{62},\sigma)=u_\tau+u_{-\tau}=u_1+u_{-1}$ is independent of the signing.

Substitution of $\mya=193/(64\sqrt2)$ and $\rr=2\sqrt2$ into
\eqref{eq:seed-u} gives the particularly short expression
\[
 \frac{u_\tau}{\sqrt2}
   =\frac{8256-4096\sqrt2\,\tau}
          {18721-10304\sqrt2\,\tau}.
\]
For clarity, the full addition is
\begin{equation}\label{eq:seed-addition}
 \begin{aligned}
 \frac{s(H_{62},\sigma)}{\sqrt2}
 &=\frac{8256-4096\sqrt2}{18721-10304\sqrt2}
   +\frac{8256+4096\sqrt2}{18721+10304\sqrt2}\\
 &=\frac{2(8256\cdot18721-2\cdot4096\cdot10304)}
          {18721^2-2\cdot10304^2}
  =\frac{140300416}{138131009}>\frac{65}{64}.
 \end{aligned}
\end{equation}

\end{proof}

\section{Why this base graph $F$ is not Ramanujan}\label{sec:scope}

In this section, we show that the constructed graph $F$ is not
Ramanujan by proving that $\lambda_2(A(F))>2\sqrt2$.
We choose two disjoint completed seeds of equal order with no
edge between them and define a vector taking opposite constant
values on these subgraphs and zero elsewhere.
This vector is orthogonal to the constant vector and has
Rayleigh quotient greater than $2\sqrt2$.
The variational characterization of $\lambda_2(A(F))$ then
gives the desired conclusion.

\begin{proposition}\label{prop:not-ramanujan}
Let $\lambda_2(A(F))$ denote the second largest eigenvalue of
$A(F)$, with eigenvalues ordered nonincreasingly and counted
with multiplicity. Then
\[
 \lambda_2(A(F)) >2\sqrt2.
\]
In particular, $F$ is not Ramanujan.
\end{proposition}
\begin{proof}

Let $C$ be the induced subgraph consisting of one copy of
$H_{62}$ and the completion vertices attached to its leaf
triples. By construction, each such triple lies entirely
within this seed. Hence
\[
  n_C:=|V(C)|
  =(3\cdot2^{63}+1)
    +2\left(\frac{3\cdot2^{62}}{3}\right)
  =2^{65}+1.
\]
Writing $o_C$ for the seed root, we have
\[
  \deg_C(v)=
  \begin{cases}
    2,&v=o_C,\\
    3,&v\in V(C)\setminus\{o_C\}.
  \end{cases}
\]
The degree-sum formula therefore gives
\begin{equation}\label{eq:completed-seed-edges}
  2|E(C)|
  =\sum_{v\in V(C)}\deg_C(v)
  =3n_C-1.
\end{equation}

Choose the two seed copies joined at the root $w$ of one
copy of $Q_1=B(H_{62},H_{62})$. Let $C_1,C_2$ be their
completed seeds, with roots $o_1,o_2$, respectively.
The vertex $w$ belongs to neither completed seed, and
\[
  V(C_1)\cap V(C_2)=\varnothing,
  \qquad
  |V(C_1)|=|V(C_2)|=n_C.
\]
There is no edge between $C_1$ and $C_2$. For each $i=1,2$,
the only edge from $C_i$ to its complement in $F$ is
$\{o_i,w\}$; see Figure~\ref{fig:completed-seeds}.

Define $\boldsymbol{x}\in\mathbb R^{V(F)}$ by
\[
  x_v=
  \begin{cases}
    (2n_C)^{-1/2},&v\in V(C_1),\\
    -(2n_C)^{-1/2},&v\in V(C_2),\\
    0,&v\notin V(C_1)\cup V(C_2).
  \end{cases}
\]
If $\mathbf 1$ denotes the vector with every coordinate
equal to one, then
\[
  \|\boldsymbol{x}\|_2^2
  =\frac{n_C+n_C}{2n_C}=1,
  \qquad
  \boldsymbol{x}^{\mathsf T}\mathbf 1
  =\frac{n_C-n_C}{\sqrt{2n_C}}=0.
\]
Edges leaving either completed seed contribute zero to
$\boldsymbol{x}^{\mathsf T}A(F)\boldsymbol{x}$, because their other endpoint has
coordinate zero. Consequently,
\[
  \begin{aligned}
    \boldsymbol{x}^{\mathsf T}A(F)\boldsymbol{x}
    &=2\sum_{\{u,v\}\in E(F)}x_ux_v\\
    &=\frac{|E(C_1)|+|E(C_2)|}{n_C}\\
    &=\frac{3n_C-1}{n_C}
     =3-\frac1{n_C},
  \end{aligned}
\]
where the third equality follows from
\eqref{eq:completed-seed-edges}.

To apply the variational characterization of $\lambda_2$,
recall that $F$ is connected and cubic. Thus
\[
  A(F)\mathbf 1=3\cdot \mathbf 1,
\]
and, for every $\boldsymbol{y}\in\mathbb R^{V(F)}$,
\[
  \boldsymbol{y}^{\mathsf T}(3I_F-A(F))\boldsymbol{y}
  =\sum_{\{u,v\}\in E(F)}(y_u-y_v)^2
  \ge0.
\]
Connectedness implies that equality holds precisely when
$\boldsymbol{y}$ is constant. Hence $3$ is the largest eigenvalue of
$A(F)$ and its eigenspace is
$\operatorname{span}\{\mathbf 1\}$. The spectral theorem
therefore gives
\[
  \begin{aligned}
    \lambda_2(A(F))
    &=\max_{\substack{\boldsymbol{y}\in\mathbb R^{V(F)}\\
                     \|\boldsymbol{y}\|_2=1,\ \boldsymbol{y}^{\mathsf T}\mathbf 1=0}}
        \boldsymbol{y}^{\mathsf T}A(F)\boldsymbol{y}\\
    &\ge \boldsymbol{x}^{\mathsf T}A(F)\boldsymbol{x}\\
    &=3-\frac1{2^{65}+1}.
  \end{aligned}
\]
Finally, $n_C>6$ and
\(
  3-\frac1{n_C}>\frac{17}{6}>2\sqrt2
\).
Since $\lambda_2(A(F))<3$ and $\lambda_2(A(F))>2\sqrt2$,
it is a nontrivial eigenvalue outside the Ramanujan interval.
Thus $F$ is not Ramanujan.
\end{proof}

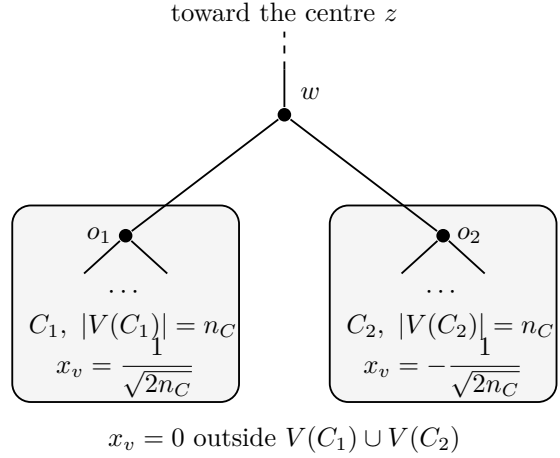
\begin{figure}[htbp]
  \centering
  \begin{tikzpicture}[
    vertex/.style={circle,fill=black,inner sep=1.8pt},
    every node/.style={font=\small},
    line width=0.7pt
  ]
    \draw[rounded corners=7pt,fill=gray!8]
      (-3.6,-1.8) rectangle (-0.6,0.8);
    \draw[rounded corners=7pt,fill=gray!8]
      (0.6,-1.8) rectangle (3.6,0.8);

    \node[vertex,label=above right:{$w$}] (w) at (0,2) {};

\node[vertex] (o1) at (-2.1,0.4) {};
\node[vertex] (o2) at (2.1,0.4) {};

\node[anchor=east,inner sep=0pt,xshift=-2pt]
  at (o1.west) {$o_1$};
\node[anchor=west,inner sep=0pt,xshift=2pt]
  at (o2.east) {$o_2$};

    \draw (w) -- (o1);
    \draw (w) -- (o2);

    \draw (w) -- (0,2.6);
    \draw[dashed] (0,2.6) -- (0,3.1);
    \node[above,align=center] at (0,3.1)
      {toward the centre $z$};

    \draw (o1) -- (-2.65,-0.1);
    \draw (o1) -- (-1.55,-0.1);
    \node at (-2.1,-0.35) {$\cdots$};

    \draw (o2) -- (1.55,-0.1);
    \draw (o2) -- (2.65,-0.1);
    \node at (2.1,-0.35) {$\cdots$};

    \node at (-2.0,-0.8) {$C_1,\,\, |V(C_1)|=n_C$};
    \node at (2.2,-0.8) {$C_2,\,\, |V(C_2)|=n_C$};

    \node at (-2.1,-1.35) {$x_v=\dfrac{1}{\sqrt{2n_C}}$};
    \node at (2.1,-1.35) {$x_v=-\dfrac{1}{\sqrt{2n_C}}$};

    \node[align=center] at (0,-2.3)
      {$x_v=0$ outside $V(C_1)\cup V(C_2)$};
  \end{tikzpicture}
  \caption{Two completed seeds $C_1,C_2$ with a common
  exterior parent $w$. Each box represents an entire induced
  subgraph, with its internal structure omitted.
  The only edge leaving $C_i$ is $\{o_i,w\}$, and there is
  no edge between $C_1$ and $C_2$. The displayed values define
  the test vector used to estimate $\lambda_2(A(F))$.}
  \label{fig:completed-seeds}
\end{figure}

\section{Conclusion}

We have constructed a finite connected simple regular graph \(F\) for which every signing has an eigenvalue outside the Ramanujan interval, disproving the Bilu--Linial signing conjecture for general regular graphs. The graph \(F\) is not Ramanujan, so this construction does not settle the conjecture restricted to Ramanujan base graphs.

Two questions remain for further study. First, does every finite connected simple regular Ramanujan graph admit a signing whose entire spectrum lies in the Ramanujan interval? Second, does every finite connected simple regular graph admit an \(\ell\)-lift, for some integer \(\ell\ge2\), whose new eigenvalues all lie in the Ramanujan interval? Here the interval is determined by the degree of the base graph, and \(\ell\) may depend on the base graph. The second question asks whether allowing more sheets can overcome the obstruction to the desired bound for \(2\)-lifts.

\vspace{0.5cm}

{\bf 
Statement on the Use of AI.} 
The counterexample and its proof strategy were developed with
the assistance of ChatGPT. The author has carefully
reviewed all AI-generated text and verified the mathematical
arguments, revising or rewriting the material as necessary.
The author takes full responsibility for the content and
correctness of this paper.

\end{document}